\documentclass{article}
\usepackage{amsmath,amsthm,amsopn,amstext,amscd,amsfonts,amssymb}
\usepackage{natbib}

\def \u {\mathop{\rm \mathcal{U}}\nolimits}

\def \tr {\mathop{\rm tr}\nolimits}
\def \re {\mathop{\rm Re}\nolimits}
\def \im {\mathop{\rm Im}\nolimits}

\def \etr {\mathop{\rm etr}\nolimits}
\def \diag {\mathop{\rm diag}\nolimits}

\newcommand {\fin}{\hfill \, $\Box$ \,\\[2ex]}

\renewenvironment{abstract}
                 {\vspace{6pt}
                  \begin{center}
                  \begin{minipage}{5in}
                  \centerline{\textbf{Abstract}}
                  \noindent\ignorespaces
                 }
                 {\end{minipage}\end{center}}
\newtheorem{thm}{\textbf{Theorem}}[section]

\newtheorem{lem}{\textbf{Lemma}}[section]

\theoremstyle{definition}
\newtheorem{defn}{\textbf{Definition}}[section]

\title{\Large \textbf{Distributions on symmetric cones II: Beta-Riesz distribution}}
\author{
  \textbf{Jos\'e A. D\'{\i}az-Garc\'{\i}a} \thanks{Corresponding author\newline
    }\\
  {\normalsize Department of Statistics and Computation} \\
  {\normalsize 25350 Buenavista, Saltillo, Coahuila, Mexico} \\
  {\normalsize E-mail: jadiaz@uaaan.mx} \\
}
\date{}
\begin{document}
\maketitle

\begin{abstract}
This article derives several properties of the Riesz distributions, such as their corresponding Bartlett
decompositions, the inverse Riesz distributions, the distribution of the generalised variance and the
distribution of their eigenvalues for real normed division algebras. In addition, introduce a kind of
generalised beta distribution termed beta-Riesz distribution for real normed division algebras. Two
versions of this distributions are proposed and some properties are studied.
\end{abstract}

\section{Introduction}\label{sec1}

It is  imminent the important role played by Wishart and beta distributions type I and II in the context
of multivariate statistics. In particular, the relationship between these two distributions to obtain the
beta distribution in terms of the distribution of two Wishart matrices.

In the last three decades, the family of elliptical contoured distributions has emerged as a robust
alternative for dealing with non-normal samples. A number of well known distributions belong to this
class, such is the case of normal, t, Bessel, Kotz type, logistic, Pearson type II and IV, among many
others; but also infinitely many new distributions can be constructed by choosing a suitable kernel
corresponding to a measurable function. Elliptical distributions are characterized by several properties,
but for the context of sampling, their large variability of kurtosis and weight of tails, can assure a
best explanation of the sample, rather than the usual forced fit to a normal model in presence of non
explicable extremal points. Another important property of this set resides in the normal invariant
statistics; i.e., assume that certain random matrix $\mathbf{X}$ follows a matrix multivariate elliptical
distribution, then, many matrices of the form  $\mathbf{Y} = f(\mathbf{X})$, for special functions
$f(\cdot)$, are invariant under the complete family of elliptical distribution, in the sense that the
distribution of  $\mathbf{Y}$ is invariant, independently of the particular distribution of $\mathbf{X}$,
in fact the distribution of $\mathbf{Y}$ coincides with the case where $\mathbf{X}$ has a matrix
multivariate normal distribution, see \citet{fz:90} and \citet{gv:93}.

However, we must note that the addressed invariance only holds under a probabilistic dependence
assumption; i.e., if the matrix
 $\mathbf{X} = \left [\begin{array}{c}
           \mathbf{X}_{1} \\
           \mathbf{X}_{2}
         \end{array}
\right]$ follows an elliptical distribution, then  $\mathbf{X}_{1}$ and $\mathbf{X}_{2}$ must be
probabilistically dependent (recall that $\mathbf{X}_{1}$ and $\mathbf{X}_{2}$ are probabilistically
independent if $\mathbf{X}$ has a matrix multivariate normal distribution, \citet{gv:93}). Now, let
$\mathbf{X}^{T}$ denotes the transpose of $\mathbf{X}$ and consider a non negative definite matrix
$\mathbf{B}$, where $\mathbf{B}^{1/2}$ denotes its non negative squared root, see \citet{fz:90}, then,
the matrix $\mathbf{F} = (\mathbf{X}^{T}_{2}\mathbf{X}_{2})^{-1/2} (\mathbf{X}^{T}_{1}\mathbf{X}_{1})
(\mathbf{X}^{T}_{2}\mathbf{X}_{2})^{-1/2}$ is said to have a matrix multivariate Beta type II
distribution, moreover, the same distribution is obtained when $\mathbf{X}$ is assumed to follow matrix
multivariate distribution, see \citet{fz:90} and \citet{gv:93}.

Now, if the random matrix $\mathbf{X}$ follows a matrix multivariate elliptical distribution and the
random matrix $\mathbf{W} = \mathbf{X}^{T}\mathbf{X}$ is defined, for every particular elliptical
distribution, $\mathbf{W}$ follows also a different distribution. The distribution of $\mathbf{W}$ is
usually known as the generalised Wishart distribution, and it heritages several properties of elliptical
contoured distributions.

By another hand some recent advances in probability has reached interesting general distributions, such
as the case of Riesz distribution, due to \citet{hl:01},  and named under the denomination of  Riesz
natural exponential family (Riesz NEF); a distribution based in an special case of the well known Riesz
measure given by \citet[p. 137]{fk:94}. In fact, Riesz distribution includes Wishart and Gamma matrix
multivariate distributions as particular cases. Now, integrating theories have also appeared in other
contexts. For example, in matrix multivariate distribution theory, some extensions from real to complex
or quaternion or octonion fields appeared separately with great theoretical effort in many cases, the
results in each field arose  unconnected each other for years; only recent a new  approach in the context
of real normed division algebras could integrate, with an unified theory, all the addressed dispersed
results in such sets of numbers. Following this tendency, recently, \citet{dg:15a} proposes two versions
of the Riesz distribution for real normed division algebras. Alternatively, \citet{dg:15b} shows that the
two versions of Riesz distributions correspond to two generalised Wishart distributions, both derived
from certain matrix multivariate elliptical distributions, which are termed matrix multivariate
Kotz-Riesz distributions.

\citet{fk:94}, and subsequently \citet{hlz:05}, propose a beta-Riesz distribution, which contains as a
special case to the beta distribution obtained in terms of the distribution Wishart, which shall be named
beta-Wishart, all this subjects in the context of simple Euclidean Jordan algebras. Such beta-Riesz
distribution is obtained analogously to the beta-Wishart distribution, but starting with a Riesz
distribution.

Based in these last two versions of the Riesz distributions, it is possible to obtain two versions of the
beta-Riesz distributions, which by analogy with the beta-Wishart distributions are termed beta-Riesz type
I. As in classical beta-Wishart distribution, in addition it is feasible to propose two version for the
beta-Riesz distribution type II. Each of the two versions for each beta-Riesz distributions of type I and
II, (both versions for each) contain as particular cases to beta-Wishart distribution type I and
beta-Wishart distribution type II, respectively.

Thus, there is no doubt about the theoretical and applied potential of Riesz distribution into the
setting in the setting of the integrative modern multivariate analysis. In general, every problem
possible ruled by a Wishart process with considerable constraints, potentially can be studied under a
more robust Riesz distribution; namely, some opportunities for consideration involve estimation of
covariance matrices and principal component estimation, under any of the classical or bayesian
approaches. Beta-Riesz distribution, for example, arises in a natural way in bayesian inference, when the
two parameter a priori distributions are probabilistically independent. To prove this fact, recall the
example about matrix $\mathbf{X}$, referred in the third paragraph of this section, and note that several
situations consider are governed by probabilistically independent matrices $\mathbf{X}_{1}$ and
$\mathbf{X}_{2}$, both of them following a Kotz-Riesz elliptical distribution. Under these assumptions,
the distribution of the corresponding random matrix $\mathbf{F}$ does not follow a beta-Wishart, but the
beta-Riesz distribution. Then, if the two parameters, $\boldsymbol{\delta}_{1}$ and
$\boldsymbol{\delta}_{2}$ are  distributed as $(\mathbf{X}^{T}_{1}\mathbf{X}_{1})$ and
$(\mathbf{X}^{T}_{2}\mathbf{X}_{2})$, respectively; i.e. $\boldsymbol{\delta}_{1}$ and
$\boldsymbol{\delta}_{2}$ follows independent  Riesz distributions, then the parameter matrix defined as
$\mathbf{F} = \boldsymbol{\delta}^{-1/2}_{2}\boldsymbol{\delta}_{1}\boldsymbol{\delta}^{-1/2}_{2}$, has a
beta-Riesz distribution.

Although during the 90's and 2000's were obtained important results in theory of random matrices
distributions, the  past 30 years have reached a substantial development. Essentially, these advances
have been archived through two approaches based on the \emph{theory of Jordan algebras} and the \emph{
theory of real normed division algebras}. A basic source of the mathematical tools of theory of random
matrices distributions under Jordan algebras can be found in \citet{fk:94}; and specifically, some works
in the context of theory of random matrices distributions based on Jordan algebras are provided in
\citet{m:94}, \citet{cl:96}, \citet{hl:01}, and \citet{hlz:05}, and the references therein. Parallel
results on theory of random matrices distributions based on real normed division algebras have been also
developed in random matrix theory and statistics, see for example \citet{gr:87}, \citet{d:02}, \citet
{f:05}, \citet{dggj:11,dggj:13}, among others. In  addition, from mathematical point of view, several
basic properties of the matrix multivariate Riesz distribution under \emph{the structure theory of normal
$j$-algebras}  and under \emph{theory of Vinberg algebras} in place of Jordan algebras have been studied,
see \citet{i:00} and \citet{bh:09}, respectively.

From a applied point of view, the relevance of \emph{the octonions} remains unclear. An excellent review
of the history, construction and many other properties of octonions is given in \citet{b:02}, where it is
stated that... \textbf{However, there is still no \emph{proof} that the octonions are useful for
understanding the real world}. We can only hope that eventually this question will be settled one way or
another."

For the sake of completeness, in this article the case of octonions is considered, but the veracity of
the results obtained for this case can only be conjectured. Nonetheless, \citet[Section 1.4.5, pp.
22-24]{f:05} it is proved that the bi-dimensional density function of the eigenvalue, for a Gaussian
ensemble of a $2 \times 2$ octonionic matrix, is obtained from the general joint density function of the
eigenvalues for the Gaussian ensemble, assuming $m = 2$ and $\beta = 8$, see Section \ref{sec2}.
Moreover, as is established in \citet{fk:94} and \citet{S:97} the result obtained in this article are
valid for the \emph{algebra of Albert}, that is when hermitian matrices ($\mathbf{S}$) or hermitian
product of matrices ($\mathbf{X}^{*}\mathbf{X}$) are $3 \times 3$ octonionic matrices.

This article studies two versions for beta-Riesz distributions type I and II for real normed division
algebras. Section \ref{sec2} reviews some definitions and notation on real normed division algebras. And
also, introduces other mathematical tools as three Jacobians with respect to Lebesgue measure and some
integral results for real normed division algebras. Section \ref{sec3} proposes diverse properties of two
versions of the Riesz distributions as their Bartlett decompositions, inverse Riesz distributions, the
distribution of the generalized variance and the distribution of their eigenvalues.  Section \ref{sec4}
introduces two generalised beta functions and, in terms of these, two beta-Riesz distributions of type I
and II are obtained for real normed division algebras. Also, the relationship between the Riesz
distributions and the beta-Riesz distributions are studied. This section concludes studying the
eigenvalues distributions of beta-Riesz distributions type I and II in their two versions for real normed
division algebras.

\section{Preliminary results}\label{sec2}

A detailed discussion of real normed division algebras may be found in \citet{b:02} and \citet{E:90}. For
convenience, we shall introduce some notation, although in general we adhere to standard notation forms.

For our purposes: Let $\mathbb{F}$ be a field. An \emph{algebra} $\mathfrak{A}$ over $\mathbb{F}$ is a
pair $(\mathfrak{A};m)$, where $\mathfrak{A}$ is a \emph{finite-dimensional vector space} over
$\mathbb{F}$ and \emph{multiplication} $m : \mathfrak{A} \times \mathfrak{A} \rightarrow A$ is an
$\mathbb{F}$-bilinear map; that is, for all $\lambda \in \mathbb{F},$ $x, y, z \in \mathfrak{A}$,
\begin{eqnarray*}
  m(x, \lambda y + z) &=& \lambda m(x; y) + m(x; z) \\
  m(\lambda x + y; z) &=& \lambda m(x; z) + m(y; z).
\end{eqnarray*}
Two algebras $(\mathfrak{A};m)$ and $(\mathfrak{E}; n)$ over $\mathbb{F}$ are said to be
\emph{isomorphic} if there is an invertible map $\phi: \mathfrak{A} \rightarrow \mathfrak{E}$ such that
for all $x, y \in \mathfrak{A}$,
$$
  \phi(m(x, y)) = n(\phi(x), \phi(y)).
$$
By simplicity, we write $m(x; y) = xy$ for all $x, y \in \mathfrak{A}$.

Let $\mathfrak{A}$ be an algebra over $\mathbb{F}$. Then $\mathfrak{A}$ is said to be
\begin{enumerate}
  \item \emph{alternative} if $x(xy) = (xx)y$ and $x(yy) = (xy)y$ for all $x, y \in \mathfrak{A}$,
  \item \emph{associative} if $x(yz) = (xy)z$ for all $x, y, z \in \mathfrak{A}$,
  \item \emph{commutative} if $xy = yx$ for all $x, y \in \mathfrak{A}$, and
  \item \emph{unital} if there is a $1 \in \mathfrak{A}$ such that $x1 = x = 1x$ for all $x \in \mathfrak{A}$.
\end{enumerate}
If $\mathfrak{A}$ is unital, then the identity 1 is uniquely determined.

An algebra $\mathfrak{A}$ over $\mathbb{F}$ is said to be a \emph{division algebra} if $\mathfrak{A}$ is
nonzero and $xy = 0_{\mathfrak{A}} \Rightarrow x = 0_{\mathfrak{A}}$ or $y = 0_{\mathfrak{A}}$ for all
$x, y \in \mathfrak{A}$.

The term ``division algebra", comes from the following proposition, which shows that, in such an algebra,
left and right division can be unambiguously performed.

Let $\mathfrak{A}$ be an algebra over $\mathbb{F}$. Then $\mathfrak{A}$ is a division algebra if, and
only if, $\mathfrak{A}$ is nonzero and for all $a, b \in \mathfrak{A}$, with $b \neq 0_{\mathfrak{A}}$,
the equations $bx = a$ and $yb = a$ have unique solutions $x, y \in \mathfrak{A}$.

In the sequel we assume $\mathbb{F} = \Re$ and consider classes of division algebras over $\Re$ or
``\emph{real division algebras}" for short.

We introduce the algebras of \emph{real numbers} $\Re$, \emph{complex numbers} $\mathfrak{C}$,
\emph{quaternions} $\mathfrak{H}$ and \emph{octonions} $\mathfrak{O}$. Then, if $\mathfrak{A}$ is an
alternative real division algebra, then $\mathfrak{A}$ is isomorphic to $\Re$, $\mathfrak{C}$,
$\mathfrak{H}$ or $\mathfrak{O}$.

Let $\mathfrak{A}$ be a real division algebra with identity $1$. Then $\mathfrak{A}$ is said to be
\emph{normed} if there is an inner product $(\cdot, \cdot)$ on $\mathfrak{A}$ such that
$$
  (xy, xy) = (x, x)(y, y) \qquad \mbox{for all } x, y \in \mathfrak{A}.
$$
If $\mathfrak{A}$ is a \emph{real normed division algebra}, then $\mathfrak{A}$ is isomorphic $\Re$,
$\mathfrak{C}$, $\mathfrak{H}$ or $\mathfrak{O}$.

There are exactly four normed division algebras: real numbers ($\Re$), complex numbers ($\mathfrak{C}$),
quaternions ($\mathfrak{H}$) and octonions ($\mathfrak{O}$), see \citet{b:02}.

Let $\mathfrak{A}$ be a division algebra over the real numbers. Then $\mathfrak{A}$ has dimension either
1, 2, 4 or 8. In other branches of mathematics, the parameters $\alpha = 2/\beta$ and $t = \beta/4$ are
used, see \citet{er:05} and \citet{k:84}, respectively.

Let ${\mathcal L}^{\beta}_{m,n}$ be the set of all $m \times n$ matrices of rank $m \leq n$ over
$\mathfrak{A}$ with $m$ distinct positive singular values, where $\mathfrak{A}$ denotes a \emph{real
finite-dimensional normed division algebra}. Let $\mathfrak{A}^{m \times n}$ be the set of all $m \times
n$ matrices over $\mathfrak{A}$. The dimension of $\mathfrak{A}^{m \times n}$ over $\Re$ is $\beta mn$.
Let $\mathbf{A} \in \mathfrak{A}^{m \times n}$, then $\mathbf{A}^{*} = \bar{\mathbf{A}}^{T}$ denotes the
usual conjugate transpose.

We denote by ${\mathfrak S}_{m}^{\beta}$ the real vector space of all $\mathbf{S} \in \mathfrak{A}^{m
\times m}$ such that $\mathbf{S} = \mathbf{S}^{*}$. In addition, let $\mathfrak{P}_{m}^{\beta}$ be the
\emph{cone of positive definite matrices} $\mathbf{S} \in \mathfrak{A}^{m \times m}$. Thus,
$\mathfrak{P}_{m}^{\beta}$ consist of all matrices $\mathbf{S} = \mathbf{X}^{*}\mathbf{X}$, with
$\mathbf{X} \in \mathfrak{L}^{\beta}_{m,n}$; then $\mathfrak{P}_{m}^{\beta}$ is an open subset of
${\mathfrak S}_{m}^{\beta}$.

Let $\mathfrak{D}_{m}^{\beta}$ consisting of all $\mathbf{D} \in \mathfrak{A}^{m \times m}$, $\mathbf{D}
= \diag(d_{1}, \dots,d_{m})$. Let $\mathfrak{T}_{U}^{\beta}(m)$ be the subgroup of all \emph{upper
triangular} matrices $\mathbf{T} \in \mathfrak{A}^{m \times m}$ such that $t_{ij} = 0$ for $1 < i < j
\leq m$.

For any matrix $\mathbf{X} \in \mathfrak{A}^{n \times m}$, $d\mathbf{X}$ denotes the\emph{ matrix of
differentials} $(dx_{ij})$. Finally, we define the \emph{measure} or volume element $(d\mathbf{X})$ when
$\mathbf{X} \in \mathfrak{A}^{m \times n}, \mathfrak{S}_{m}^{\beta}$, and $\mathfrak{D}_{m}^{\beta}$, see
\citet{d:02} and \citet{dggj:11}.

If $\mathbf{X} \in \mathfrak{A}^{m \times n}$ then $(d\mathbf{X})$ (the Lebesgue measure in
$\mathfrak{A}^{m \times n}$) denotes the exterior product of the $\beta mn$ functionally independent
variables
$$
  (d\mathbf{X}) = \bigwedge_{i = 1}^{m}\bigwedge_{j = 1}^{n}dx_{ij} \quad \mbox{ where }
    \quad dx_{ij} = \bigwedge_{k = 1}^{\beta}dx_{ij}^{(k)}.
$$

If $\mathbf{S} \in \mathfrak{S}_{m}^{\beta}$ (or $\mathbf{S} \in \mathfrak{T}_{U}^{\beta}(m)$ with
$t_{ii} >0$, $i = 1, \dots,m$) then $(d\mathbf{S})$ (the Lebesgue measure in $\mathfrak{S}_{m}^{\beta}$
or in $\mathfrak{T}_{U}^{\beta}(m)$) denotes the exterior product of the exterior product of the
$m(m-1)\beta/2 + m$ functionally independent variables,
$$
  (d\mathbf{S}) = \bigwedge_{i=1}^{m} ds_{ii}\bigwedge_{i < j}^{m}\bigwedge_{k = 1}^{\beta}
                      ds_{ij}^{(k)}.
$$
Observe, that for the Lebesgue measure $(d\mathbf{S})$ defined thus, it is required that $\mathbf{S} \in
\mathfrak{P}_{m}^{\beta}$, that is, $\mathbf{S}$ must be a non singular Hermitian matrix (Hermitian
definite positive matrix).

If $\mathbf{\Lambda} \in \mathfrak{D}_{m}^{\beta}$ then $(d\mathbf{\Lambda})$ (the Legesgue measure in
$\mathfrak{D}_{m}^{\beta}$) denotes the exterior product of the $\beta m$ functionally independent
variables
$$
  (d\mathbf{\Lambda}) = \bigwedge_{i = 1}^{n}\bigwedge_{k = 1}^{\beta}d\lambda_{i}^{(k)}.
$$

Now, we show three Jacobians in terms of the $\beta$ parameter, which are proposed as extensions of real,
complex or quaternion cases, see \citet{dggj:11}.

\begin{lem}\label{lemhlt} Let $\mathbf{X}$ and $\mathbf{Y} \in
\mathfrak{P}_{m}^{\beta}$ matrices of functionally independent variables, and let $\mathbf{Y} =
\mathbf{AXA^{*}} + \mathbf{C}$, where $\mathbf{A} \in {\mathcal L}_{m,m}^{\beta} $ and $\mathbf{C} \in
\mathfrak{P}_{m}^{\beta}$ are matrices of constants. Then
\begin{equation}\label{hlt}
    (d\mathbf{Y}) = |\mathbf{A}^{*}\mathbf{A}|^{(m-1)\beta/2+1} (d\mathbf{X}),
\end{equation}
where $|\mathbf{B}|$ denotes the determinant of $\mathbf{B}$.
\end{lem}

\begin{lem} [Cholesky's decomposition]\label{lemch}
Let $\mathbf{S} \in \mathfrak{P}_{m}^{\beta}$ and $\mathbf{T} \in \mathfrak{T}_{U}^{\beta}(m)$ with
$t_{ii} > 0$, $i = 1, 2, \ldots , m$. Then
\begin{equation}\label{ch}
    (d\mathbf{S}) = \left\{
                      \begin{array}{ll}
                        2^{m} \displaystyle\prod_{i = 1}^{m} t_{ii}^{(m - i)\beta + 1} (d\mathbf{T})
& \hbox{if } \ \mathbf{S} = \mathbf{T}^{*}\mathbf{T};  \\
                        2^{m} \displaystyle\prod_{i = 1}^{m} t_{ii}^{(i - 1)\beta + 1} (d\mathbf{T})
&\hbox{if } \ \mathbf{S} = \mathbf{TT}^{*}.
                      \end{array}
    \right.
\end{equation}
\end{lem}

\begin{lem}\label{lemi}
Let $\mathbf{S} \in \mathfrak{P}_{m}^{\beta}.$ Then, ignoring the sign, if $\mathbf{Y} = \mathbf{S}^{-1}+
\mathbf{C}$, $\mathbf{C} \in \mathfrak{P}_{m}^{\beta}$ is a matrix of constants,
\begin{equation}\label{i}
    (d\mathbf{Y}) = |\mathbf{S}|^{-\beta(m - 1) - 2}(d\mathbf{S}).
\end{equation}
\end{lem}

Next is stated a general result, that is useful in a variety of situations, which enable us to transform
the density function of a matrix $\mathbf{X} \in \mathfrak{P}_{m}^{\beta}$ to the density function of its
eigenvalues, see \citet{dggj:11}.

\begin{lem}\label{eig}
Let $\mathbf{X} \in \mathfrak{P}_{m}^{\beta}$ be a random matrix with a density function
$f_{_{\mathbf{X}}}(\mathbf{X})$. Then the joint density function of the eigenvalues $\lambda_{1}, \dots,
\lambda_{m}$ of $\mathbf{X}$ is
\begin{equation}\label{dfeig}
    \frac{\pi^{m^{2}\beta/2+ \varrho}}{\Gamma_{m}^{\beta}[m\beta/2]} \prod_{i<
    j}^{m}(\lambda_{i} -\lambda_{j})^{\beta}\int_{\mathbf{H} \in
    \mathfrak{U}^{\beta}(m)}f(\mathbf{HLH}^{*})(d\mathbf{H})
\end{equation}
where $\mathbf{L} = \diag(\lambda_{1}, \dots, \lambda_{m})$, $\lambda_{1}> \cdots > \lambda_{m}> 0$,
$(d\mathbf{H})$ is the normalised Haar measure, $\Gamma_{m}^{\beta}[a]$ denotes the Gamma function for
the space $\mathfrak{S}^{\beta}_{m}$ \citep{gr:87} and
$$
  \varrho = \left\{
             \begin{array}{rl}
               0, & \beta = 1; \\
               -m, & \beta = 2; \\
               -2m, & \beta = 4; \\
               -4m, & \beta = 8.
             \end{array}
           \right.
$$
\end{lem}

Finally, let's recall the multidimensional convolution theorem in terms of the Laplace transform. For
this purpose,  let's use the complexification $\mathfrak{S}_{m}^{\beta, \mathfrak{C}} =
\mathfrak{S}_{m}^{\beta} + i \mathfrak{S}_{m}^{\beta}$ of $\mathfrak{S}_{m}^{\beta}$. That is,
$\mathfrak{S}_{m}^{\beta, \mathfrak{C}}$ consist of all matrices $\mathbf{Z} \in
(\mathfrak{F^{\mathfrak{C}}})^{m \times m}$ of the form $\mathbf{Z} = \mathbf{X} + i\mathbf{Y}$, with
$\mathbf{X}, \mathbf{Y} \in \mathfrak{S}_{m}^{\beta}$. It comes to $\mathbf{X} = \re(\mathbf{Z})$ and
$\mathbf{Y} = \im(\mathbf{Z})$ as the \emph{real and imaginary parts} of $\mathbf{Z}$, respectively. The
\emph{generalised right half-plane} $\mathbf{\Phi}_{m}^{\beta} = \mathfrak{P}_{m}^{\beta} + i
\mathfrak{S}_{m}^{\beta}$ in $\mathfrak{S}_{m}^{\beta,\mathfrak{C}}$ consists of all $\mathbf{Z} \in
\mathfrak{S}_{m}^{\beta,\mathfrak{C}}$ such that $\re(\mathbf{Z}) \in \mathfrak{P}_{m}^{\beta}$, see
\cite[p. 801]{gr:87}.

\begin{defn}
If $f(\mathbf{X})$ is a function of $\mathbf{X} \in \mathfrak{P}_{m}^{\beta}$, the Laplace transform of
$f(\mathbf{X})$ is defined to be
\begin{equation}\label{LT}
    g(\mathbf{T}) = \int_{\mathbf{X} \in \mathfrak{P}_{m}^{\beta}}
  \etr\{-\mathbf{XZ}\}f(\mathbf{X})(d\mathbf{X}).
\end{equation}
where $\mathbf{Z} \in \mathbf{\Phi}_{m}^{\beta}$ and $\etr(\cdot) = \exp(\tr(\cdot))$.
\end{defn}

\begin{lem}\label{lemC}
If $g_{1}(\mathbf{Z})$ and $g_{2}(\mathbf{Z})$ are the respective Laplace transforms of the densities
$f_{_{\mathbf{X}}}(\mathbf{X})$ and $g_{_{\mathbf{Y}}}(\mathbf{Y})$ then the product
$g_{1}(\mathbf{Z})g_{2}(\mathbf{Z})$ is the Laplace transform of the convolution
$f_{_{\mathbf{X}}}(\mathbf{X})\ast g_{_{\mathbf{Y}}}(\mathbf{Y})$, where
\begin{equation}\label{conv}
    h_{_{\mathbf{\Xi}}} (\mathbf{\Xi}) = f_{_{\mathbf{X}}}(\mathbf{X})\ast g_{_{\mathbf{Y}}}(\mathbf{Y})
    = \int_{\mathbf{0} < \mathbf{\Upsilon} < \mathbf{\Xi}} f_{_{\mathbf{X}}}(\mathbf{\Upsilon})
   g_{_{\mathbf{Y}}}(\mathbf{\Xi}-\mathbf{\Upsilon})(d\mathbf{\Upsilon}),
\end{equation}
with $\mathbf{\Xi} = \mathbf{X}+\mathbf{Y}$ and $\mathbf{\Upsilon} = \mathbf{X}$.
\end{lem}

\section{Riesz distributions}\label{sec3}

This section shows two versions of the Riesz distributions \citep{dg:15a} and the study of their Bartlett
decompositions. Also,  inverse Riesz distributions and the joint density of its eigenvalues are obtained.

\begin{defn}\label{defnRd}
Let $\mathbf{\Sigma} \in \mathbf{\Phi}_{m}^{\beta}$ and  $\kappa = (k_{1}, k_{2}, \dots, k_{m}) \in
\Re^{m}$.
\begin{enumerate}
  \item Then it is said that $\mathbf{X}$ has a Riesz distribution of type I if its density function is
  \begin{equation}\label{dfR1}
    \frac{\beta^{am+\sum_{i = 1}^{m}k_{i}}}{\Gamma_{m}^{\beta}[a,\kappa] |\mathbf{\Sigma}|^{a}q_{\kappa}(\mathbf{\Sigma})}
    \etr\{-\beta\mathbf{\Sigma}^{-1}\mathbf{X}\}|\mathbf{X}|^{a-(m-1)\beta/2 - 1}
    q_{\kappa}(\mathbf{X})(d\mathbf{X})
  \end{equation}
  for $\mathbf{X} \in \mathfrak{P}_{m}^{\beta}$ and $\re(a) \geq (m-1)\beta/2 - k_{m}$; where $\Gamma_{m}^{\beta}[a,\kappa]$
  is the generalised gamma function of weight $\kappa$ and $q_{\kappa}(\mathbf{A})$ is the highest wight vector or generalised
  power of $\mathbf{A}$\citep[see][]{gr:87}; denoting this fact as $\mathbf{X} \sim \mathfrak{R}^{\beta, I}_{m}(a,\kappa,
  \mathbf{\Sigma})$.
  \item Then it is said that $\mathbf{X}$ has a Riesz distribution of type II if its density function is
  \begin{equation}\label{dfR2}
     \frac{\beta^{am-\sum_{i = 1}^{m}k_{i}}}{\Gamma_{m}^{\beta}[a,-\kappa]
   |\mathbf{\Sigma}|^{a}q_{\kappa}(\mathbf{\Sigma}^{-1})}\etr\{-\beta\mathbf{\Sigma}^{-1}\mathbf{X}\}
  |\mathbf{X}|^{a-(m-1)\beta/2 - 1} q_{\kappa}(\mathbf{X}^{-1}) (d\mathbf{X})
  \end{equation}
  for $\mathbf{X} \in \mathfrak{P}_{m}^{\beta}$ and $\re(a) > (m-1)\beta/2 + k_{1}$; where $\Gamma_{m}^{\beta}[a,-\kappa]$
  is the generalised gamma function of weight $\kappa$ proposed by \citet{k:66};
  denoting this fact as $\mathbf{X} \sim \mathfrak{R}^{\beta, II}_{m}(a,\kappa,
  \mathbf{\Sigma})$.
\end{enumerate}
\end{defn}

\begin{thm}\label{teo1}
Let $\mathbf{\Sigma} \in \mathbf{\Phi}_{m}^{\beta}$ and  $\kappa = (k_{1}, k_{2}, \dots, k_{m}) \in
\Re^{m}$. And let $\mathbf{Y} = \mathbf{X}^{-1}$.
\begin{enumerate}
  \item Then if $\mathbf{X}$ has a Riesz distribution of type I the density of $\mathbf{Y}$ is
  \begin{equation}\label{dfIR1}
    \frac{\beta^{am+\sum_{i = 1}^{m}k_{i}}}{\Gamma_{m}^{\beta}[a,\kappa] |\mathbf{\Sigma}|^{a}q_{\kappa}(\mathbf{\Sigma})}
    \etr\{-\beta\mathbf{\Sigma}^{-1}\mathbf{Y}^{-1}\}|\mathbf{Y}|^{-(a+(m-1)\beta/2 + 1)}
    q_{\kappa}(\mathbf{Y}^{-1})(d\mathbf{Y})
  \end{equation}
  for $\re(a) \geq (m-1)\beta/2 - k_{m}$ and is termed as \emph{inverse Riesz distribution of type
  I}.
  \item Then if $\mathbf{X}$ has a Riesz distribution of type II the density of $\mathbf{Y}$ is
  \begin{equation}\label{dfIR2}
     \frac{\beta^{am-\sum_{i = 1}^{m}k_{i}}}{\Gamma_{m}^{\beta}[a,-\kappa]
   |\mathbf{\Sigma}|^{a}q_{\kappa}(\mathbf{\Sigma}^{-1})}\etr\{-\beta\mathbf{\Sigma}^{-1}\mathbf{Y}^{-1}\}
  |\mathbf{Y}|^{-(a+(m-1)\beta/2 + 1)} q_{\kappa}(\mathbf{Y}) (d\mathbf{Y})
  \end{equation}
  for $\re(a) > (m-1)\beta/2 + k_{1}$ and it said that $\mathbf{Y}$ has a \emph{inverse Riesz distribution of type
  II}.
\end{enumerate}
\end{thm}
\textbf{Proof. } It is immediately noted that $(d\mathbf{X}) = |\mathbf{Y}|^{-\beta(m-1)-2}(d\mathbf{Y})$
and from (\ref{dfR1}) and (\ref{dfR2}).\fin

Note that, the density function (\ref{dfIR1}) was studied previously by \citet{tz:12} in real case.

Observe that, if $\kappa = (0,0, \dots,0)$ in two densities in Definition \ref{defnRd} and Theorem
\ref{teo1} the matrix multivariate gamma  and inverse gamma distributions are obtained. As consequence,
in this last case if $\mathbf{\Sigma} = 2\mathbf{\Sigma}$ and $a = \beta n/2$, the Wishart and inverse
Wishart distributions are obtained, too.

\begin{thm}\label{teodB}
Let $\mathbf{T}\in \mathfrak{T}_{U}^{\beta}(m)$ with $t_{ii} > 0$, $i = 1, 2, \ldots , m$ and define
$\mathbf{X} = \mathbf{T}^{*}\mathbf{T}$.
\begin{enumerate}
  \item If $\mathbf{X}$ has a Riesz distribution of type I, (\ref{dfR1}), with $\mathbf{\Sigma}
   = \mathbf{I}_{m}$, then the elements $t_{ij}
  \ (1\leq i\leq j\leq m)$ of $\mathbf{T}$ are all independent. Furthermore,  $t^{2}_{ii} \sim
  \mathcal{G}^{\beta}(a+k_{i}-(i-1)\beta/2,1)$ and $\sqrt{2}t_{ij} \sim \mathcal{N}^{\beta}_{1}(0,1) \
  (1\leq i < j\leq m)$.
  \item If $\mathbf{X}$ has a Riesz distribution of type II, (\ref{dfR2}), with $\mathbf{\Sigma}
   = \mathbf{I}_{m}$, then the elements $t_{ij}
  \ (1\leq i\leq j\leq m)$ of $\mathbf{T}$ are all independent. Moreover,  $t^{2}_{ii} \sim
  \mathcal{G}^{\beta}(a-k_{i}-(i-1)\beta/2,1)$ and $\sqrt{2}t_{ij} \sim \mathcal{N}^{\beta}_{1}(0,1) \
  (1\leq i < j\leq m)$.
\end{enumerate}
Where $x \sim \mathcal{G}^{\beta}(a,\alpha)$ denotes a gamma distribution with parameters $a$ and
$\alpha$ and $y \sim \mathcal{N}^{\beta}_{1}(0,1)$ denotes a random variable with standard normal
distribution for real normed division algebras. Moreover, their respective densities are
$$
  \mathcal{G}^{\beta}(x:a,\alpha)=\frac{1}{(\alpha/\beta)^{a}\Gamma[a]}\exp\{-\beta x/\alpha\} x^{a-1} (dx),
$$
and
$$
  \mathcal{N}^{\beta}_{1}(y:0,1) = \frac{1}{(2\pi/\beta)^{\beta/2}} \exp\{-\beta y^{2}/2\}(dy)
$$
where $x \in \mathfrak{L}^{\beta}_{1,1}$, $y \in \mathfrak{P}_{1}^{\beta}$, Re$(a) > 0$ and $\alpha \in
\mathbf{\Phi}_{1}^{\beta}$, see \citet{dggj:11}.
\end{thm}
\textbf{Proof. } This is given for the case of Riesz distribution type I. The proof for Riesz
distribution type II is the same thing.  The density of $\mathbf{X}$ is
\begin{equation}\label{eqch}
    \frac{\beta^{am+\sum_{i = 1}^{m}k_{i}}}{\Gamma_{m}^{\beta}[a,\kappa]}
    \etr\{-\beta\mathbf{X}\}|\mathbf{X}|^{a-(m-1)\beta/2 - 1}
    q_{\kappa}(\mathbf{X})(d\mathbf{X}).
\end{equation}
Since $\mathbf{X} = \mathbf{T}^{*}\mathbf{T}$ we have
\begin{eqnarray*}
  \tr \mathbf{X} &=& \tr \mathbf{T}^{*}\mathbf{T} = \displaystyle\sum_{i \leq j}^{m} t_{ij}^{2},\\
  |\mathbf{X}| &=& |\mathbf{T}^{*}\mathbf{T}| = |\mathbf{T}|^{2} = \displaystyle\prod_{i = 1}^{m} t_{ii}^{2},\\
  q_{\kappa}(\mathbf{X}) &=& q_{\kappa}(\mathbf{T}^{*}\mathbf{T}) =
  |\mathbf{T}^{*}\mathbf{T}|^{k_{m}} \displaystyle\prod_{i =
  1}^{m-1}|\mathbf{T}_{i}^{*}\mathbf{T}_{i}|^{k_{i}-k_{i+1}} =  \displaystyle\prod_{i = 1}^{m}
  t_{ii}^{2k_{i}},
\end{eqnarray*}
and by Theorem \ref{lemch} noting that $dt_{ii}^{2} = 2 t_{ii}dt_{ii}$, then
\begin{eqnarray*}
  (d\mathbf{X}) &=& \displaystyle 2^{m}\prod_{i = 1}^{m} t_{ii}^{\beta(m-i)+1}\left(\bigwedge_{i\leq
  j}dt_{ij}\right), \\
   &=& \displaystyle \prod_{i = 1}^{m} \left(t_{ii}^{2}\right)^{\beta(m-i)/2}\left(\bigwedge_{i
   = 1}dt_{ii}^{2}\right)\wedge\left(\bigwedge_{i < j}dt_{ij}\right).
\end{eqnarray*}
Substituting this expression in (\ref{eqch}) we find that the joint density of the $t_{ij} \ (1\leq i\leq
j\leq m)$ can be written as
$$
  \prod_{i=1}^{m} \frac{\beta^{a+k_{i}-(i-1)\beta/2}}{\Gamma[a+k_{i}-(i-1)\beta/2]}
  \exp\{-\beta t_{ii}^{2}\} \left(t_{ii}^{2}\right)^{a+k_{i}-(i-1)\beta/2-1}\left(dt_{ii}^{2}\right)
$$
$$
  \times \ \prod_{i <j}^{m}\frac{1}{(\pi/\beta)^{\beta/2}} \exp\{- \beta t_{ij}^{2}\}(dt_{ij}),
$$
only observe that
\begin{eqnarray*}
  \frac{\beta^{am+\sum_{i = 1}^{m}k_{i}}}{\Gamma_{m}^{\beta}[a,\kappa]} &=&
  \frac{\beta^{am+\sum_{i = 1}^{m}k_{i}-m(m-1)\beta/4}}{\beta^{-m(m-1)\beta/4}
  \pi^{m(m-1)\beta/4} \prod_{i=1}^{m}\Gamma[a+k_{i}-(i-1)\beta/2]} \\
   &=& \prod_{i=1}^{m}\frac{\beta^{a+k_{i}-(i-1)\beta/2}}{\Gamma[a+k_{i}-(i-1)\beta/2]}
       \prod_{i<j}^{m}\frac{1}{(\pi /\beta)^{\beta/2}}.
\end{eqnarray*}
\fin In analogy to generalised variance for Wishart case, the following result gives the distribution of
$|\mathbf{X}|$ when $\mathbf{X}$ has a Riesz distribution type I or type II.

\begin{thm}
Let $v = |\mathbf{X}|/|\mathbf{\Sigma}|$. Then
\begin{enumerate}
  \item if $\mathbf{X}$ has a Riesz distribution of type I, (\ref{dfR1}), the density of $v$ is
  $$
    \prod_{i=1}^{m}\mathcal{G}^{\beta}\left(t_{ii}^{2}:a+k_{i}-(i-1)\beta/2,1\right).
  $$
  \item if $\mathbf{X}$ has a Riesz distribution of type II, (\ref{dfR2}), the density of $v$ is
  $$
  \prod_{i=1}^{m}\mathcal{G}^{\beta}\left(t_{ii}^{2}:a-k_{i}-(i-1)\beta/2,1\right).
  $$
\end{enumerate}
where $t_{ii}^{2}$, $i = 1, \dots, m$, are independent random variables.
\end{thm}
\textbf{Proof. } This is immediately from Theorem \ref{teodB}, noting that if
$$
  \mathbf{B} =\mathbf{\Sigma}^{-1/2}\mathbf{X\Sigma}^{-1/2} = \mathbf{T}^{*}\mathbf{T},
$$
with $\mathbf{T}\in \mathfrak{T}_{U}^{\beta}(m)$ and $t_{ii} > 0$, $i = 1, 2, \ldots , m$, then
$$
  |\mathbf{B}| = \prod_{i=1}^{m}t_{ii}^{2} = |\mathbf{X}|/|\mathbf{\Sigma}| = v.
$$
\fin

\begin{thm}\label{teoE} Let $\mathbf{\Sigma}= \mathbf{I}_{m}$ and $\kappa =
(k_{1}, k_{2}, \dots, k_{m})$, $k_{1}\geq k_{2}\geq \cdots \geq k_{m} \geq 0$, $k_{1}, k_{2},\dots,
k_{m}$ are nonnegative integers.
\begin{enumerate}
  \item Let $\lambda_{1}, \dots,\lambda_{m}$,  $\lambda_{1}> \cdots >\lambda_{m} > 0$ be
  the eigenvalues of $\mathbf{X}$. Then if $\mathbf{X}$ has a Riesz distribution of type I,
  the joint density  of $\lambda_{1}, \dots, \lambda_{m}$ is
  $$
     \frac{\beta^{am+\sum_{i = 1}^{m}k_{i}} \ \pi^{m^{2}\beta/2+ \varrho}}{\Gamma_{m}^{\beta}[m\beta/2]\Gamma_{m}^{\beta}[a,\kappa]
     } \prod_{i< j}^{m}(\lambda_{i} -\lambda_{j})^{\beta}
     \exp\left\{-\beta \sum_{i =1}^{m}\lambda_{i}\right\}
     \hspace{4cm}
  $$
  \begin{equation}\label{dER1}\hspace{4cm}
     \times \ \prod_{i=1}^{m}\lambda_{i}^{a-(m-1)\beta/2 - 1} \frac{C_{\kappa}^{\beta}(\mathbf{L})}{C_{\kappa}^{\beta}(\mathbf{I}_{m})}.
  \end{equation}
  where $\mathbf{L} = \diag(\lambda_{1}, \dots,\lambda_{m})$ and $\re(a) \geq (m-1)\beta/2 - k_{m}$.
  \item Let $\delta_{1}, \dots,\delta_{m}$,  $\delta_{1}> \cdots >\delta_{m} > 0$ be
  the eigenvalues of $\mathbf{X}$. Then if $\mathbf{X}$ has a Riesz distribution of type II,
  the joint density of their eigenvalues is
  $$
     \frac{\beta^{am-\sum_{i = 1}^{m}k_{i}} \ \pi^{m^{2}\beta/2+ \varrho}}
     {\Gamma_{m}^{\beta}[m\beta/2]\Gamma_{m}^{\beta}[a,\kappa]}
     \prod_{i< j}^{m}(\delta_{i} -\delta_{j})^{\beta}
     \exp\left\{-\beta \sum_{i =1}^{m}\delta_{i}\right\}
     \hspace{4cm}
  $$
  \begin{equation}\label{dER2}\hspace{4cm}
     \times \ \prod_{i=1}^{m}\delta_{i}^{a-(m-1)\beta/2 - 1} \frac{C_{\kappa}^{\beta}(\mathbf{D}^{-1})}{C_{\kappa}^{\beta}(\mathbf{I}_{m})}).
  \end{equation}
  where $\mathbf{D} = \diag(\delta_{1}, \dots,\delta_{m})$, $\re(a) > (m-1)\beta/2 + k_{1}$.
\end{enumerate}
Where $\varrho$ is defined in Lemma 3 and $C_{\kappa}^{\beta}(\cdot)$ denotes the zonal spherical
functions or spherical polynomials, see \citet{gr:87} and \citet[Chapter XI, Section 3]{fk:94}.
\end{thm}
\textbf{Proof. } 1. From Lemma 3
$$
    \frac{\beta^{am+\sum_{i = 1}^{m}k_{i}} \ \pi^{m^{2}\beta/2+ \varrho}}{\Gamma_{m}^{\beta}[m\beta/2]\Gamma_{m}^{\beta}[a,\kappa]
    |\mathbf{\Sigma}|^{a}q_{\kappa}(\mathbf{\Sigma})} \prod_{i<
    j}^{m}(\lambda_{i} -\lambda_{j})^{\beta} \hspace{4cm}
$$
$$\hspace{2cm}
    \int_{\mathbf{H} \in \mathfrak{U}^{\beta}(m)}\etr\{-\beta \mathbf{HLH}^{*}\}
    |\mathbf{HLH}^{*}|^{a-(m-1)\beta/2 - 1} q_{\kappa}(\mathbf{HLH}^{*})(d\mathbf{H}).
$$
Therefore,
$$
    \frac{\beta^{am+\sum_{i = 1}^{m}k_{i}} \ \pi^{m^{2}\beta/2+ \varrho}}{\Gamma_{m}^{\beta}[m\beta/2]\Gamma_{m}^{\beta}[a,\kappa]
    |\mathbf{\Sigma}|^{a}q_{\kappa}(\mathbf{\Sigma})} \prod_{i<
    j}^{m}(\lambda_{i} -\lambda_{j})^{\beta} \prod_{i=1}^{m}\lambda_{i}^{a-(m-1)\beta/2 - 1}
    \exp\left\{-\beta \sum_{i =1}^{m}\lambda_{i}\right\}
    \hspace{2cm}
$$
$$\hspace{5.5cm}
    \int_{\mathbf{H} \in \mathfrak{U}^{\beta}(m)}
     q_{\kappa}(\mathbf{HLH}^{*})(d\mathbf{H}),
$$
the result is follow from \citep[Equation 4.8(2) and Definition 5.3]{gr:87} and \citet[Chapter XI,
Section 3]{fk:94}.

2. Is proved similarly.\fin

\section{Generalised beta distributions: Beta-Riesz distributions.}\label{sec4}

This section defines several versions for the beta functions and their relation with the gamma functions
type I and II. In these terms, the beta-Riesz distributions type I and II are defined. Finally, diverse
properties are studied.

\subsection{Generalised $c$-beta function}
A generalised of \emph{multivariate beta function} for the cone $\mathfrak{P}^{\beta}_{m}$, denoted as
$\mathcal{B}_{m}^{\beta}[a,\kappa;b, \tau]$, can be defined as
\begin{equation}\label{gcbeta1}
    \int_{\mathbf{0}<\mathbf{S}<\mathbf{I}_{m}}
    |\mathbf{S}|^{a-(m-1)\beta/2-1} q_{\kappa}(\mathbf{S})|\mathbf{I}_{m} - \mathbf{S}|^{b-(m-1)\beta/2-1}
    q_{\tau}(\mathbf{I}_{m} - \mathbf{S})(d\mathbf{S})
\end{equation}
where $\kappa = (k_{1}, k_{2}, \dots, k_{m}) \in \Re^{m}$, $\tau = (t_{1}, t_{2}, \dots, t_{m}) \in
\Re^{m}$, Re$(a) > (m-1)\beta/2-k_{m}$ and Re$(b)> (m-1)\beta/2-t_{m}$. This is defined by \citet[p.
130]{fk:94} for Euclidean simple Jordan algebras. In the context of multivariate analysis, this
generalised beta function can be termed \textit{generalised $c$-beta function type I}, as analogy to the
correspondence case of matrix multivariate beta distribution, and using the term $c$-beta as abbreviation
of classical-beta. In the next theorem we introduce the \textit{generalised $c$-beta function type II}
and its relation with the generalised gamma function.

\begin{thm}\label{teo2}
The generalised $c$-beta function type I can be expressed as
$$
  \int_{\mathbf{R} \in\mathfrak{P}_{m}^{\beta}} |\mathbf{R}|^{a-(m-1)\beta/2-1}
    q_{\kappa}(\mathbf{R})|\mathbf{I}_{m} + \mathbf{R}|^{-(a+b)} q_{-(\kappa+\tau)}(\mathbf{I}_{m} + \mathbf{R})
    (d\mathbf{R})
$$
$$
   \hspace{7cm} = \frac{\Gamma_{m}^{\beta}[a,\kappa] \Gamma_{m}^{\beta}[b,\tau]}{\Gamma_{m}^{\beta}[a+b,
    \kappa+\tau]},
$$
where $\kappa = (k_{1}, k_{2}, \dots, k_{m}) \in \Re^{m}$, $\tau = (t_{1}, t_{2}, \dots, t_{m}) \in
\Re^{m}$, Re$(a)> (m-1)\beta/2-k_{m}$ and Re$(b)> (m-1)\beta/2-t_{m}$. The integral expression is termed
\emph{generalised $c$-beta function type II}.
\end{thm}
\textbf{Proof. } Let $\u(\mathbf{I}_{m} - \mathbf{S})^{*} \u(\mathbf{I}_{m} - \mathbf{S})=
(\mathbf{I}_{m} - \mathbf{S})$ the Cholesky decomposition of $(\mathbf{I}_{m} - \mathbf{S})$ where
$\u(\mathbf{I}_{m} - \mathbf{S}) \in \mathfrak{T}_{U}^{\beta}(m)$ and define $\mathbf{R} =
\u(\mathbf{I}_{m} - \mathbf{S})^{*-1} \mathbf{S} \u(\mathbf{I}_{m} - \mathbf{S})^{-1}$ then
\begin{eqnarray*}
  \mathbf{R} &=& \u(\mathbf{I}_{m} - \mathbf{S})^{*-1} (\mathbf{I}_{m}-(\mathbf{I}_{m}-\mathbf{S})) \u(\mathbf{I}_{m} -
\mathbf{S})^{-1} \\
   &=& \u(\mathbf{I}_{m} - \mathbf{S})^{*-1}\u(\mathbf{I}_{m} - \mathbf{S})^{-1} - \mathbf{I}_{m}
\end{eqnarray*}
Thus $(\mathbf{I}_{m} + \mathbf{R}) = \u(\mathbf{I}_{m} - \mathbf{S})^{*-1}\u(\mathbf{I}_{m} -
\mathbf{S})^{-1}$. By Lemma \ref{lemi}
\begin{eqnarray*}
  (d\mathbf{R}) &=& |\u(\mathbf{I}_{m} - \mathbf{S})^{*-1}\u(\mathbf{I}_{m} -
    \mathbf{S})^{-1}|^{(m-1)\beta+2} (d\mathbf{S}) \\
   &=& |\mathbf{I}_{m} + \mathbf{R}|^{(m-1)\beta+2} (d\mathbf{S}),
\end{eqnarray*}
therefore $(d\mathbf{S}) = (\mathbf{I}+\mathbf{R})^{-(m-1)\beta-2} (d\mathbf{R})$. Now remember that
$q_{\kappa}(\mathbf{T}^{*-1}\mathbf{A}\mathbf{T}^{-1}) = q_{\kappa}(\mathbf{A}) q_{-\kappa}(\mathbf{B}) =
q_{\kappa}(\mathbf{A}) q_{\kappa}^{-1}(\mathbf{B})$ for $\mathbf{B} = \mathbf{T}^{*}\mathbf{T}$, we have
$$
  |\mathbf{R}|^{a-(m-1)\beta/2-1}q_{\kappa}(\mathbf{R}) = |\mathbf{I}_{m} - \mathbf{S}|^{-a+(m-1)\beta/2+1}
  |\mathbf{S}|^{a-(m-1)\beta/2-1}q_{\kappa}(\mathbf{S}) q_{-\kappa}(\mathbf{I}_{m} - \mathbf{S})
$$
and
$$
  |\mathbf{I}_{m} + \mathbf{R}|^{b-(m-1)\beta/2-1}q_{\tau}(\mathbf{I} + \mathbf{R}) = |\mathbf{I}_{m} -
  \mathbf{S}|^{-b+(m-1)\beta/2+1} q_{-\kappa}(\mathbf{I}_{m} - \mathbf{S}),
$$
then
$$
  |\mathbf{I}_{m} + \mathbf{R}|^{-b+(m-1)\beta/2+1}q_{-\tau}(\mathbf{I} + \mathbf{R}) = |\mathbf{I}_{m} -
  \mathbf{S}|^{b-(m-1)\beta/2-1} q_{\kappa}(\mathbf{I}_{m} - \mathbf{S}).
$$
from where the desired result is obtained.

For the expression in terms of generalised gamma function, let $\mathbf{B} =
\u(\mathbf{\Xi})^{*}\mathbf{S}\u(\mathbf{\Xi})$ in (\ref{gcbeta1}), such that $\mathbf{\Xi} =
\u(\mathbf{\Xi})^{*}\u(\mathbf{\Xi})$. Then $(d\mathbf{S}) = |\mathbf{\Xi}|^{-(m - 1)\beta/2 - 1}
(d\mathbf{B})$, and
$$
  \mathcal{B}_{m}^{\beta}[a,\kappa;b, \tau]  |\mathbf{\Xi}|^{a+b-(m - 1)\beta/2 - 1} q_{\kappa +\tau}(\mathbf{\Xi})\hspace{6cm}
$$
$$\hspace{2cm}
  = \int_{\mathbf{0}}^{\mathbf{\Xi}} |\mathbf{B}|^{a-(m-1)\beta/2 -1} q_{\kappa}(\mathbf{B})
  |\mathbf{\Xi} -\mathbf{B}|^{b-(m-1)\beta/2 -1} q_{\tau}(\mathbf{\Xi}-\mathbf{B})
  (d\mathbf{B}).
$$
Taking Laplace transform of both size, by (\ref{dfR1}), the left size is
$$
  \int_{\mathbf{\Xi} \in \mathfrak{P}_{m}^{\beta}} \mathcal{B}_{m}^{\beta}[a,\kappa;b, \tau]
  \etr\{-\mathbf{\Xi Z}\}|\mathbf{\Xi}|^{a+b-(m - 1)\beta/2 - 1} q_{\kappa +\tau}(\mathbf{\Xi})(d\mathbf{\Xi})
  \hspace{3cm}
$$
$$\hspace{4cm}
  = \mathcal{B}_{m}^{\beta}[a,\kappa;b, \tau] \Gamma_{m}^{\beta}[a+b; \kappa +\tau]|\mathbf{Z}|^{-(a+b)}
  q_{\kappa + \tau}(\mathbf{Z}^{-1}),
$$
and applying Lemma \ref{lemC}, $ g_{1}(\mathbf{Z})$ is
$$
   \int_{\mathbf{\Xi} \in \mathfrak{P}_{m}^{\beta}}
  \etr\{-\mathbf{\Xi Z}\}|\mathbf{\Xi}|^{a-(m - 1)\beta/2 - 1} q_{\kappa}(\mathbf{\Xi})(d\mathbf{\Xi})
  =  \Gamma_{m}^{\beta}[a; \kappa]|\mathbf{Z}|^{-a} q_{\kappa}(\mathbf{Z}^{-1}),
$$
and $ g_{2}(\mathbf{Z})$ is given by
$$
  \int_{\mathbf{B} \in \mathfrak{P}_{m}^{\beta}}
  \etr\{-\mathbf{BZ}\}|\mathbf{B}|^{b-(m - 1)\beta/2 - 1} q_{\kappa}(\mathbf{B})(d\mathbf{B})
  =  \Gamma_{m}^{\beta}[b; \tau]|\mathbf{Z}|^{-b} q_{\tau}(\mathbf{Z}^{-1}).
$$
Thus, equally
$$
  \mathcal{B}_{m}^{\beta}[a,\kappa;b, \tau] = \frac{\Gamma_{m}^{\beta}[a,\kappa] \Gamma_{m}^{\beta}[b,\tau]}
  {\Gamma_{m}^{\beta}[a+b, \kappa+\tau]}.
$$
\fin

\subsection{Generalised $k$-beta function}

Alternatively, a generalised of \emph{multivariate beta function} for the cone
$\mathfrak{P}^{\beta}_{m}$, can be defined and denoted as
\begin{equation}\label{gkbeta1}
    \mathcal{B}_{m}^{\beta}[a,-\kappa;b, -\tau] =\int_{\mathbf{0}<\mathbf{S}<\mathbf{I}_{m}}
    |\mathbf{S}|^{a-(m-1)\beta/2-1} q_{\kappa}\left(\mathbf{S}^{-1}\right)|\mathbf{I}_{m} - \mathbf{S}|^{b-(m-1)\beta/2-1}
    q_{\tau}\left((\mathbf{I}_{m} - \mathbf{S})^{-1}\right)(d\mathbf{S})
\end{equation}
where $\kappa = (k_{1}, k_{2}, \dots, k_{m}) \in \Re^{m}$, $\tau = (t_{1}, t_{2}, \dots, t_{m}) \in
\Re^{m}$, Re$(a)> (m-1)\beta/2+k_{1}$ and Re$(b)> (m-1)\beta/2+t_{1}$. Again, in the context of
multivariate analysis, this generalised $k$-beta function can be termed \textit{generalised $k$-beta
function type I}, as an analogy to the corresponding case of matrix multivariate beta distribution and
using the term $k$-beta as abbreviation of Khatri-beta. Next theorem introduces the \textit{generalised
$k$-beta function type II} and its relation with the generalised gamma function proposed by \citet{k:66}.

\begin{thm}\label{teo3}
The generalised $k$-beta function type II can be expressed as
$$
  \int_{\mathbf{R} \in\mathfrak{P}_{m}^{\beta}} |\mathbf{R}|^{a-(m-1)\beta/2-1}
    q_{\kappa}(\mathbf{R}^{-1})|\mathbf{I}_{m} + \mathbf{R}|^{-(a+b)} q_{-(\kappa+\tau)}
    \left((\mathbf{I}_{m} + \mathbf{R})^{-1}\right)(d\mathbf{R})
$$
$$
   \hspace{7cm} = \frac{\Gamma_{m}^{\beta}[a,-\kappa] \Gamma_{m}^{\beta}[b,-\tau]}{\Gamma_{m}^{\beta}[a+b,
    -\kappa-\tau]},
$$
where $\kappa = (k_{1}, k_{2}, \dots, k_{m}) \in \Re^{m}$, $\tau = (t_{1}, t_{2}, \dots, t_{m}) \in
\Re^{m}$, Re$(a) > (m-1)\beta/2+k_{1}$ and Re$(b)> (m-1)\beta/2+t_{1}$. The integral expression is termed
\emph{generalised $k$-beta function type II}.
\end{thm}
\textbf{Proof. } The proof is analogous to the given for Theorem \ref{teo2}. \fin

Observe that if $\kappa = (0, \dots, 0) \in \Re^{m}$ and $\tau = (0, \dots, 0) \in \Re^{m}$ in
(\ref{gcbeta1}), Theorem \ref{teo2}, (\ref{gkbeta1}) and Theorem \ref{teo3} the classical beta function
is obtained, see \citet{h:55}.

\subsection{$c$-beta-Riesz and $k$-beta-Riesz distributions}

As an immediate consequence of the results of the previous section, next the $c$-beta-Riesz and
$k$-beta-Riesz distributions types I and II are defined.

\begin{defn}\label{defncBRd}
Let $\kappa = (k_{1}, k_{2}, \dots, k_{m}) \in \Re^{m}$ and $\tau = (t_{1}, t_{2}, \dots, t_{m}) \in
\Re^{m}$.
\begin{enumerate}
  \item Then it said that $\mathbf{S}$ has a \emph{$c$-beta-Riesz distribution of type I} if its density function is
  \begin{equation}\label{dfcbR1}
    \frac{1}{\mathcal{B}_{m}^{\beta}[a,\kappa;b, \tau]}
    |\mathbf{S}|^{a-(m-1)\beta/2-1} q_{\kappa}(\mathbf{S})|\mathbf{I}_{m} - \mathbf{S}|^{b-(m-1)\beta/2-1}
    q_{\tau}(\mathbf{I}_{m} - \mathbf{S})(d\mathbf{S}),
  \end{equation}
  where $\mathbf{0}<\mathbf{S}<\mathbf{I}_{m}$ and $Re(a) > (m-1)\beta/2-k_{m}$ and Re$(b)> (m-1)\beta/2-t_{m}$.
  \item Then it said that $\mathbf{R}$ has a \emph{$c$-beta-Riesz distribution of type II} if its density function is
  \begin{equation}\label{dfcbR2}
     \frac{1}{\mathcal{B}_{m}^{\beta}[a,\kappa;b, \tau]} |\mathbf{R}|^{a-(m-1)\beta/2-1}
     q_{\kappa}(\mathbf{R})|\mathbf{I}_{m} + \mathbf{R}|^{-(a+b)} q_{-(\kappa+\tau)}(\mathbf{I}_{m} + \mathbf{R})
     (d\mathbf{R}),
  \end{equation}
  where $\mathbf{R} \in\mathfrak{P}_{m}^{\beta}$ and Re$(a) > (m-1)\beta/2-k_{m}$ and Re$(b)> (m-1)\beta/2-t_{m}$.
\end{enumerate}
\end{defn}
Similarly we have
\begin{defn}\label{defnkBRd}
Let $\kappa = (k_{1}, k_{2}, \dots, k_{m}) \in \Re^{m}$ and $\tau = (t_{1}, t_{2}, \dots, t_{m}) \in
\Re^{m}$.
\begin{enumerate}
  \item Then it said that $\mathbf{S}$ has a \emph{$k$-beta-Riesz distribution of type I} if its density function is
  \begin{equation}\label{dfkbR1}
    \frac{1}{\mathcal{B}_{m}^{\beta}[a,-\kappa;b, -\tau]}
    |\mathbf{S}|^{a-(m-1)\beta/2-1} q_{\kappa}(\mathbf{S}^{-1})|\mathbf{I}_{m} - \mathbf{S}|^{b-(m-1)\beta/2-1}
    q_{\tau}\left((\mathbf{I}_{m} - \mathbf{S})^{-1}\right)(d\mathbf{S}),
  \end{equation}
  where $\mathbf{0}<\mathbf{S}<\mathbf{I}_{m}$ and  Re$(a) > (m-1)\beta/2+k_{1}$ and Re$(b)> (m-1)\beta/2+t_{1}$.
  \item Then it said that $\mathbf{R}$ has a \emph{$k$-beta-Riesz distribution of type II} if its density function is
  \begin{equation}\label{dfkbR2}
     \frac{1}{\mathcal{B}_{m}^{\beta}[a,-\kappa;b, -\tau]} |\mathbf{R}|^{a-(m-1)\beta/2-1}
     q_{\kappa}(\mathbf{R}^{-1})|\mathbf{I}_{m} + \mathbf{R}|^{-(a+b)} q_{-(\kappa+\tau)}\left((\mathbf{I}_{m} +
     \mathbf{R})^{-1}\right)
     (d\mathbf{R}),
  \end{equation}
  where $\mathbf{R} \in\mathfrak{P}_{m}^{\beta}$ and  Re$(a) > (m-1)\beta/2+k_{1}$ and Re$(b)> (m-1)\beta/2+t_{1}$.
\end{enumerate}
\end{defn}

Observe that the relationship between the densities (\ref{dfcbR1}) and (\ref{dfcbR2}), and between the
densities (\ref{dfkbR1}) and (\ref{dfkbR2}) are easily obtained from the theorems \ref{teo2} and
\ref{teo3}, respectively.

The following result state the relation between the Riesz and beta-Riesz distributions.

\begin{thm}\label{teo4}\label{teobb1}
Let $\mathbf{X}_{1}$ and $\mathbf{X}_{2}$ be independently distributed as Riesz distribution type I, such
that $\mathbf{X}_{1} \sim \mathfrak{R}^{\beta, I}_{m}(a,\kappa, \mathbf{\Sigma})$ and $\mathbf{X}_{2}
\sim \mathfrak{R}^{\beta, I}_{m}(b,\tau, \mathbf{\Sigma})$, Re$(a) > (m-1)\beta/2+k_{1}$ and Re$(b)>
(m-1)\beta/2+t_{1}$. Let
$$
  \mathbf{S} =\u(\mathbf{X}_{1}+\mathbf{X}_{2})^{*-1}\mathbf{X}_{1}\u(\mathbf{X}_{1}+
  \mathbf{X}_{2})^{-1},
$$
where $\u(\mathbf{X}_{1}+\mathbf{X}_{2}) \in \mathfrak{T}_{U}^{\beta}(m)$ is such that
$(\mathbf{X}_{1}+\mathbf{X}_{2}) = \u(\mathbf{X}_{1}+\mathbf{X}_{2})^{*}
\u(\mathbf{X}_{1}+\mathbf{X}_{2})$. Then $\mathbf{S}$ and $(\mathbf{X}_{1}+\mathbf{X}_{2})$ are
independent, $\mathbf{S}$ has a $c$-beta-Riesz distribution type I and $(\mathbf{X}_{1}+\mathbf{X}_{2})
\sim \mathfrak{R}^{\beta, I}_{m}(a+b,\kappa+\tau, \mathbf{I}_{m})$.
\end{thm}
\textbf{Proof. } The joint density of $\mathbf{X}_{1}$ and $\mathbf{X}_{2}$ is given by
$$
  \frac{\beta^{(a+b)m+\sum_{i = 1}^{m}(k_{i}+t_{i})}}{\Gamma_{m}^{\beta}[a,\kappa] \Gamma_{m}^{\beta}[b,\tau]
  |\mathbf{\Sigma}|^{a+b}q_{\kappa+\tau}(\mathbf{\Sigma})}
    \etr\{-\beta\mathbf{\Sigma}^{-1}(\mathbf{X}_{1}+\mathbf{X}_{2})\}|\mathbf{X}_{1}|^{a-(m-1)\beta/2 - 1}
    q_{\kappa}(\mathbf{X}_{1})
$$
$$
  \times
    |\mathbf{X}_{2}|^{b-(m-1)\beta/2 - 1}
    q_{\tau}(\mathbf{X}_{2})(d\mathbf{X}_{1})\wedge(d\mathbf{X}_{2}).
$$
Let $\mathbf{Y} = \mathbf{X}_{1}+\mathbf{X}_{2}$ and $\mathbf{Z} = \mathbf{X}_{1}$, then,
$(d\mathbf{X}_{1})\wedge(d\mathbf{X}_{2}) = (d\mathbf{Y})\wedge(d\mathbf{Z})$. Then the joint density of
$\mathbf{Y}$ and $\mathbf{Z}$ is given by
$$
  \frac{\beta^{(a+b)m+\sum_{i = 1}^{m}(k_{i}+t_{i})}}{\Gamma_{m}^{\beta}[a,\kappa] \Gamma_{m}^{\beta}[b,\tau]
  |\mathbf{\Sigma}|^{a+b}q_{\kappa+\tau}(\mathbf{\Sigma})}
    \etr\{-\beta \mathbf{\Sigma}^{-1}\mathbf{Y}\}|\mathbf{Z}|^{a-(m-1)\beta/2 - 1}
    q_{\kappa}(\mathbf{Z})
$$
$$
  \times
    |\mathbf{Y}-\mathbf{Z}|^{b-(m-1)\beta/2 - 1}
    q_{\tau}(\mathbf{Y}-\mathbf{Z})(d\mathbf{Y})\wedge(d\mathbf{Z}).
$$
Let $\mathbf{W} = \u(\mathbf{Y})^{*}\u(\mathbf{Y})$, with $\u(\mathbf{Y}) \in
\mathfrak{T}_{U}^{\beta}(m)$ and $\mathbf{Z} = \u(\mathbf{Y})^{*} \mathbf{S} \u(\mathbf{Y})$. Observing
that $\u(\mathbf{Y})$  is a function  of $\mathbf{W}$
$$
  (d\mathbf{Y})\wedge(d\mathbf{Z}) = |\u(\mathbf{Y})^{*} \u(\mathbf{Y})|^{\beta(m-1)/2+1}
  (d\u(\mathbf{Y})^{*} \u(\mathbf{Y}))\wedge(d\mathbf{S})
$$
Hence the joint density of $\mathbf{S}$ and $\mathbf{W} = \u(\mathbf{Y})^{*}\u(\mathbf{Y})$ is
$$
  \frac{\beta^{(a+b)m+\sum_{i = 1}^{m}(k_{i}+t_{i})}}{\Gamma_{m}^{\beta}[a+b,\kappa+\tau]
  |\mathbf{\Sigma}|^{a+b}q_{\kappa+\tau}(\mathbf{\Sigma})}
    \etr\{-\beta \mathbf{\Sigma}^{-1}\u(\mathbf{Y})^{*}\u(\mathbf{Y})\} |\u(\mathbf{Y})^{*}\u(\mathbf{Y})|^{a+b-\beta(m-1)/2-1}
$$
$$
   \hspace{8cm} q_{\kappa+\tau}(\u(\mathbf{Y})^{*}\u(\mathbf{Y})) (d\u(\mathbf{Y})^{*}\u(\mathbf{Y}))
$$
$$
  \times
  \frac{\Gamma_{m}^{\beta}[a+b,\kappa+\tau]}{\Gamma_{m}^{\beta}[a,\kappa] \Gamma_{m}^{\beta}[b,\tau]}
    |\mathbf{S}|^{a-(m-1)\beta/2 - 1} q_{\kappa}(\mathbf{S}) |\mathbf{I}- \mathbf{S}|^{b-(m-1)\beta/2 - 1}
    q_{\tau}(\mathbf{I}- \mathbf{S})(d\mathbf{S}).
$$
which shows that $\mathbf{W} = \u(\mathbf{Y})^{*}\u(\mathbf{Y}) = \mathbf{X}_{1}+\mathbf{X}_{2} \sim
\mathfrak{R}^{\beta, I}_{m}(a+b,\kappa+\tau, \mathbf{\Sigma})$ independently of $\mathbf{S}$ with a
$c$-beta-Riesz distribution type I. \fin

\begin{thm}\label{teo5}\label{teobb2}
Let $\mathbf{X}_{1}$ and $\mathbf{X}_{2}$ be independently distributed as Riesz distribution type I, such
that $\mathbf{X}_{1} \sim \mathfrak{R}^{\beta, I}_{m}(a,\kappa, \mathbf{\Sigma})$ and $\mathbf{X}_{2}
\sim \mathfrak{R}^{\beta, I}_{m}(b,\tau, \mathbf{\Sigma})$, Re$(a) > (m-1)\beta/2+k_{1}$ and Re$(b)>
(m-1)\beta/2+t_{1}$. Let
$$
  \mathbf{R} = \u(\mathbf{X}_{2})^{*-1}\mathbf{X}_{1}\u(\mathbf{X}_{2})^{-1},
$$
where $\u(\mathbf{X}_{2}) \in \mathfrak{T}_{U}^{\beta}(m)$ is such that $\mathbf{X}_{1} =
\u(\mathbf{X}_{2})^{*}\u(\mathbf{X}_{2})$. Then $\mathbf{S}$ has a $c$-beta-Riesz distribution type II.
\end{thm}
\textbf{Proof. } From Theorem \ref{teo2} we know that if $\mathbf{S}$ has a $c$-beta-Riesz distribution
type I then $\mathbf{R} = \u(\mathbf{I}_{m} - \mathbf{S})^{*-1} \mathbf{S} \u(\mathbf{I}_{m} -
\mathbf{S})^{-1}$ has a $c$-beta-Riesz distribution type II. In addition the theorem establish that if
$\mathbf{X}_{1}$ and $\mathbf{X}_{2}$ be independently distributed as Riesz distribution type I, such
that $\mathbf{X}_{1} \sim \mathfrak{R}^{\beta, I}_{m}(a,\kappa, \mathbf{\Sigma})$ and $\mathbf{X}_{2}
\sim \mathfrak{R}^{\beta, I}_{m}(b,\tau, \mathbf{\Sigma})$ then $\mathbf{R} =
\u(\mathbf{X}_{1})^{*-1}\mathbf{X}_{1} \u(\mathbf{X}_{2})^{-1}$. Thus, the desired result is follow if we
proof that
$$
  \mathbf{R} = \u(\mathbf{I}_{m} - \mathbf{S})^{*-1} \mathbf{S} \u(\mathbf{I}_{m} -
  \mathbf{S})^{-1} = \u(\mathbf{X}_{2})^{*-1}\mathbf{X}_{1}\u(\mathbf{X}_{2})^{-1}.
$$
With this aim in mind, let $\u(\mathbf{Y}) \in \mathfrak{T}_{U}^{\beta}(m)$, such that $\mathbf{X}_{1} +
\mathbf{X}_{2} = \u(\mathbf{Y})^{*}\u(\mathbf{Y})$, then $\mathbf{S} =
\u(\mathbf{Y})^{*-1}\mathbf{X}_{1}\u(\mathbf{Y})^{-1}$. Now, if $\mathbf{X}_{2} =
\u(\mathbf{X}_{2})^{*}\u(\mathbf{X}_{2})$, with $\u(\mathbf{X}_{2}) \in \mathfrak{T}_{U}^{\beta}(m)$. We
have
\begin{eqnarray*}
  \mathbf{I}_{m}-\mathbf{S} &=& \mathbf{I}_{m} - \u(\mathbf{Y})^{*-1}\mathbf{X}_{1}\u(\mathbf{Y})^{-1}\\
   &=& \u(\mathbf{Y})^{*-1}(\u(\mathbf{Y})^{*}\u(\mathbf{Y})- \mathbf{X}_{1})\u(\mathbf{Y})^{-1} \\
   &=& \u(\mathbf{Y})^{*-1}\mathbf{X}_{2}\u(\mathbf{Y})^{-1} \\
   &=& \u(\mathbf{Y})^{*-1}\u(\mathbf{X}_{2})^{*}\u(\mathbf{X}_{2})\u(\mathbf{Y})^{-1} \\
   &=& \left(\u(\mathbf{X}_{2})\u(\mathbf{Y})^{-1}\right)^{*}\left(\u(\mathbf{X}_{2})\u(\mathbf{Y})^{-1}\right)\\
   &=& \u(\mathbf{I}_{m} - \mathbf{S})^{*} \u(\mathbf{I}_{m} - \mathbf{S}).
\end{eqnarray*}
This least equally is obtained observing that $\left(\u(\mathbf{X}_{2}) \u(\mathbf{Y})^{-1}\right) \in
\mathfrak{T}_{U}^{\beta}(m)$, then $\left(\u(\mathbf{X}_{2}) \u(\mathbf{Y})^{-1}\right) =
\u(\mathbf{I}_{m} - \mathbf{S})$. Therefore
\begin{eqnarray*}
  \u(\mathbf{I}_{m} - \mathbf{S})^{*-1} \mathbf{S} \u(\mathbf{I}_{m} -
  \mathbf{S})^{-1} &=& \left(\u(\mathbf{X}_{2})\u(\mathbf{Y})^{-1}\right)^{*-1}
  \mathbf{S} \left(\u(\mathbf{X}_{2})\u(\mathbf{Y})^{-1}\right)^{-1} \\
   &=& \u(\mathbf{X}_{2})^{*-1}\u(\mathbf{Y})^{*} \mathbf{S} \u(\mathbf{Y})\u(\mathbf{X}_{2})^{-1} \\
   &=& \u(\mathbf{X}_{2})^{*-1}\mathbf{X}_{1} \u(\mathbf{X}_{2})^{-1}.
\end{eqnarray*}
From where the desired result is obtained. \fin

The following theorems \ref{teo6} and \ref{teo7} contain versions  for $k$-beta-Riesz distributions of
theorems \ref{teobb1} and \ref{teobb2}, whose proofs are similar.

\begin{thm}\label{teo6}
Let $\mathbf{X}_{1}$ and $\mathbf{X}_{2}$ be independently distributed as Riesz distribution type II,
such that $\mathbf{X}_{1} \sim \mathfrak{R}^{\beta, II}_{m}(a,\kappa, \mathbf{\Sigma})$ and
$\mathbf{X}_{2} \sim \mathfrak{R}^{\beta, II}_{m}(b,\tau, \mathbf{\Sigma})$, Re$(a) > (m-1)\beta/2+k_{1}$
and Re$(b)> (m-1)\beta/2+t_{1}$. Let
$$
  \mathbf{S} =\u(\mathbf{X}_{1}+\mathbf{X}_{2})^{*-1}\mathbf{X}_{1}\u(\mathbf{X}_{1}+
  \mathbf{X}_{2})^{-1},
$$
where $\u(\mathbf{X}_{1}+\mathbf{X}_{2}) \in \mathfrak{T}_{U}^{\beta}(m)$ is such that
$(\mathbf{X}_{1}+\mathbf{X}_{2})=
\u(\mathbf{X}_{1}+\mathbf{X}_{2})^{*}\u(\mathbf{X}_{1}+\mathbf{X}_{2})$. Then $\mathbf{S}$ has a
$k$-beta-Riesz distribution type I.
\end{thm}

\begin{thm}\label{teo7}
Let $\mathbf{X}_{1}$ and $\mathbf{X}_{2}$ be independently distributed as Riesz distribution type I, such
that $\mathbf{X}_{1} \sim \mathfrak{R}^{\beta, II}_{m}(a,\kappa, \mathbf{\Sigma})$ and $\mathbf{X}_{2}
\sim \mathfrak{R}^{\beta, II}_{m}(b,\tau, \mathbf{\Sigma})$, Re$(a) > (m-1)\beta/2+k_{1}$ and Re$(b)>
(m-1)\beta/2+t_{1}$. Let
$$
  \mathbf{R} = \u(\mathbf{X}_{2})^{*-1}\mathbf{X}_{1}\u(\mathbf{X}_{2})^{-1},
$$
where $\u(\mathbf{X}_{2})  \in \mathfrak{T}_{U}^{\beta}(m)$ is such that $\mathbf{X}_{2} =
\u(\mathbf{X}_{1})^{*}\u(\mathbf{X}_{1})$. Then $\mathbf{S}$ has a $k$-beta-Riesz distribution type II.
\end{thm}

\subsection{Some properties of the $c$-beta-Riesz and $k$-beta-Riesz distributions}

This section derives the distributions of eigenvalues for $c$-beta-Riesz and $k$-beta-Riesz distributions
type I and II. First consider the following integrals:
$$
  Q(\kappa,\tau, \mathbf{A},\mathbf{B}) = \int_{\mathfrak{U}^{\beta}(m)}q_{\kappa}(\mathbf{HAH}^{*})q_{\tau}
  \left(\mathbf{H}\mathbf{B}\mathbf{H}^{*}\right)
  (d\mathbf{H})
$$
and
$$
  Q_{1}(\kappa,\tau, \mathbf{A},\mathbf{B}) = \int_{\mathfrak{U}^{\beta}(m)}q_{\kappa}(\mathbf{HAH}^{*})q_{-(\kappa+\tau)}
  \left(\mathbf{H}\mathbf{B}\mathbf{H}^{*}\right)
  (d\mathbf{H})
$$

\begin{thm}\label{teo8}
Let $\mathbf{\Sigma} \in \mathbf{\Phi}_{m}^{\beta}$,  $\kappa = (k_{1}, k_{2}, \dots, k_{m})$, $k_{1}\geq
k_{2}\geq \cdots \geq k_{m} \geq 0$, $k_{1}, k_{2},\dots, k_{m}$ are nonnegative integers and $\tau =
(t_{1}, t_{2}, \dots, t_{m})$, $t_{1}\geq t_{2}\geq \cdots \geq t_{m} \geq 0$, $t_{1}, t_{2},\dots,
t_{m}$ are nonnegative integers.
\begin{enumerate}
  \item Let $\mathbf{L} = \diag(\lambda_{1}, \dots,\lambda_{m})$,  $\lambda_{1}> \cdots >\lambda_{m} > 0$ be
  the eigenvalues of $\mathbf{S}$. Then if $\mathbf{S}$ has a $c$-beta-Riesz distribution of type I,
  the joint density  of $\lambda_{1}, \dots, \lambda_{m}$ is
  $$
    \frac{\pi^{m^{2}\beta/2+ \varrho}}{\Gamma_{m}^{\beta}[m\beta/2]\mathcal{B}_{m}^{\beta}[a,\kappa;b, \tau]}
    \prod_{i< j}^{m}(\lambda_{i} -\lambda_{j})^{\beta} \prod_{i=1}^{m}\lambda_{i}^{a-(m-1)\beta/2-1}
    \prod_{i=1}^{m}(1 - \lambda_{i})^{b-(m-1)\beta/2-1}\hspace{3cm}
$$
$$
    \hspace{8cm}
    Q(\kappa,\tau, \mathbf{L},\mathbf{I}_{m}-\mathbf{L})
    \left(\bigwedge_{i=1}^{m}d\lambda_{i}\right),
  $$
  where $0 < \lambda_{i}<1$, $i = 1, \dots,m$ and $Re(a) > (m-1)\beta/2-k_{m}$ and Re$(b)> (m-1)\beta/2-t_{m}$.

  \item Let $\Delta = \diag\left(\delta_{1}, \dots,\delta_{m}\right)$,  $\delta_{1}> \cdots >\delta_{m} > 0$ be
  the eigenvalues of $\mathbf{R}$. Then if $\mathbf{R}$ has a $c$-beta-Riesz distribution of type II,
  the joint density of their eigenvalues is
  $$
    \frac{\pi^{m^{2}\beta/2+ \varrho}}{\Gamma_{m}^{\beta}[m\beta/2]\mathcal{B}_{m}^{\beta}[a,\kappa;b, \tau]}
    \prod_{i< j}^{m}(\delta_{i} -\delta_{j})^{\beta} \prod_{i=1}^{m}\delta_{i}^{a-(m-1)\beta/2-1}
    \prod_{i=1}^{m}(1 - \delta_{i})^{-(a+b)}
    \hspace{3cm}
  $$
  $$
    \hspace{8cm}
     Q_{1}\left(\kappa,\tau, \mathbf{\Delta},\left(\mathbf{I}_{m} + \mathbf{\Delta}\right)\right)
    \left(\bigwedge_{i=1}^{m}d\delta_{i}\right),
  $$
  where $\delta_{i} > 0$, $i = 1, \dots,m$ and Re$(a) > (m-1)\beta/2-k_{m}$ and Re$(b)>
  (m-1)\beta/2-t_{m}$.
\end{enumerate}
$\varrho$ is defined in Lemma \ref{eig}.
\end{thm}
\textbf{Proof. } This is due to applying the Lemma \ref{eig} in (\ref{dfcbR1}) and (\ref{dfcbR2}). \fin

This section conclude establishing the Theorem \ref{teo8} for the case of the $k$-beta-Riesz
distributions.

\begin{thm}\label{teo9}
Let $\mathbf{\Sigma} \in \mathbf{\Phi}_{m}^{\beta}$,  $\kappa = (k_{1}, k_{2}, \dots, k_{m})$, $k_{1}\geq
k_{2}\geq \cdots \geq k_{m} \geq 0$, $k_{1}, k_{2},\dots, k_{m}$ are nonnegative integers and $\tau =
(t_{1}, t_{2}, \dots, t_{m})$, $t_{1}\geq t_{2}\geq \cdots \geq t_{m} \geq 0$, $t_{1}, t_{2},\dots,
t_{m}$ are nonnegative integers.
\begin{enumerate}
  \item Let $\mathbf{L} = \diag\left(\lambda_{1}, \dots,\lambda_{m}\right)$,  $\lambda_{1}> \cdots >\lambda_{m} > 0$ be
  the eigenvalues of $\mathbf{S}$. Then if $\mathbf{S}$ has a $k$-beta-Riesz distribution of type I,
  the joint density  of $\lambda_{1}, \dots, \lambda_{m}$ is
  $$
    \frac{\pi^{m^{2}\beta/2+ \varrho}}{\Gamma_{m}^{\beta}[m\beta/2]\mathcal{B}_{m}^{\beta}[a,-\kappa;b, -\tau]}
    \prod_{i< j}^{m}(\lambda_{i} -\lambda_{j})^{\beta} \prod_{i=1}^{m}\lambda_{i}^{a-(m-1)\beta/2-1}
    \prod_{i=1}^{m}(1 - \lambda_{i})^{b-(m-1)\beta/2-1}
    \hspace{3cm}
  $$
  $$
    \hspace{7cm}
    Q\left(\kappa,\tau, \mathbf{L}^{-1},(\mathbf{I}_{m}-\mathbf{L})^{-1}\right)
    \left(\bigwedge_{i=1}^{m}d\lambda_{i}\right),
  $$
  where $0 < \lambda_{i}<1$, $i = 1, \dots,m$ and Re$(a) > (m-1)\beta/2+k_{1}$ and Re$(b)> (m-1)\beta/2+t_{1}$.

  \item Let $\Delta = \diag\left(\delta_{1}, \dots,\delta_{m}\right)$,  $\delta_{1}> \cdots >\delta_{m} > 0$ be
  the eigenvalues of $\mathbf{R}$. Then if $\mathbf{R}$ has a $k$-beta-Riesz distribution of type II,
  the joint density of their eigenvalues is
  $$
    \frac{\pi^{m^{2}\beta/2+ \varrho}}{\Gamma_{m}^{\beta}[m\beta/2]\mathcal{B}_{m}^{\beta}[a,-\kappa;b, -\tau]}
    \prod_{i< j}^{m}(\delta_{i} -\delta_{j})^{\beta} \prod_{i=1}^{m}\delta_{i}^{a-(m-1)\beta/2-1}
    \prod_{i=1}^{m}(1 - \delta_{i})^{-(a+b)}
    \hspace{3cm}
  $$
  $$
    \hspace{6cm}
     Q_{1}\left(\kappa,\tau, \mathbf{\Delta}^{-1},\left(\mathbf{I}_{m} + \mathbf{\Delta}\right)^{-1}\right)
    \left(\bigwedge_{i=1}^{m}d\delta_{i}\right),
  $$
  where $\delta_{i} > 0$, $i = 1, \dots,m$ and Re$(a) > (m-1)\beta/2+k_{1}$ and Re$(b)> (m-1)\beta/2+t_{1}$.
\end{enumerate}
$\varrho$ is defined in Lemma \ref{eig}.
\end{thm}

Finally observe that if in all result of this section are taking  $\kappa = (0,0, \dots,0) \in \Re^{m}$
and $\tau = (0,0, \dots,0) \in \Re^{m}$ the obtained results are the corresponding to matrix multivariate
beta distributions of type I and II.

\section*{Conclusions}

Finally, note that the real dimension of real normed division algebras can be expressed as powers of 2,
$\beta = 2^{n}$ for $n = 0,1,2,3$. On the other hand, as observed from \citet{k:84}, the results obtained
in this work can be extended to hypercomplex cases; that is, for complex, bicomplex, biquaternion and
bioctonion (or sedenionic) algebras, which of course are not division algebras (except the complex
algebra). Also note, that hypercomplex algebras are obtained by replacing the real numbers with complex
numbers in the construction of real normed division algebras. Thus, the results for hypercomplex algebras
are obtained by simply replacing $\beta$ with $2\beta$ in our results. Alternatively, following
\citet{k:84}, it can be concluded that, results are true for `$2^{n}$-ions', $n = 0,1,2,3,4,5$,
emphasising that only for $n=0,1,2,3$ are the result algebras, in fact, real normed division algebras.

\section*{Acknowledgements}
The author wish to thank the Editor and the anonymous reviewers for their constructive comments on the
preliminary version of this paper. This article was written under the existing research agreement between
the first author and the Universidad Aut\'onoma Agraria Antonio Narro, Saltillo, M\'exico.

\end{document}